\documentclass[12pt]{amsart}
\usepackage[margin=1in]{geometry}
\usepackage{graphicx}
\usepackage{turnstile}
\usepackage{amssymb}
\usepackage{amsmath}
\usepackage{amsfonts}
\usepackage{amstext}
\usepackage{amsthm}
\usepackage{setspace}
\usepackage{array}
\usepackage{turnstile}
\usepackage{textcomp}
\usepackage{tipa}
\usepackage{esvect}
\usepackage{graphicx}

\numberwithin{equation}{section}

\newtheorem{Lem}{Lemma}

\newtheorem{Thm}{Theorem}
\newtheorem*{claim}{Claim}
\newtheorem{definition}{Definition}
\newtheorem{remark}{Remark}
\newtheorem{corollary}{Corollary}

\newcommand{\cl}[1]{\ensuremath{\overline{#1}}}
\newcommand{\set}[1]{\ensuremath{\{#1\}}}

\newcommand{\E}[1]{\ensuremath{\mathbb{E}\left[ #1 \right]}}
\newcommand{\Prob}[1]{\ensuremath {P\left( #1 \right)}}

\newcommand{\surj}{\ensuremath{\twoheadrightarrow}}

\newcommand{\curly}[1]{\ensuremath{\mathcal #1}}

\newcommand{\ep}{\ensuremath{\epsilon}}

\newcommand{\C}{\ensuremath{\mathbb C}}
\newcommand{\w}{\ensuremath{\omega}}
\newcommand{\W}{\ensuremath{\Omega}}

\newcommand{\Di}[2]{\frac{\partial #2}{\partial #1}}

\newcommand{\DDi}[2]{\frac{\partial^2 #2}{\partial #1^2}}
\newcommand{\twiddle}[1]{\ensuremath{\widetilde{#1}}}

\newcommand{\lr}[1]{\ensuremath{\left( #1 \right)}}

\newcommand{\setst}[2]{\ensuremath{\left\{#1\,\middle|\,#2\right\}}}

\newcommand{\abs}[1]{\left\lvert #1 \right\rvert}
\newcommand{\norm}[1]{\left\lVert#1\right\rVert}

\newcommand{\inprod}[2]{\ensuremath{\left\langle#1,#2\right\rangle}}

\newcommand{\gives}{\ensuremath{\rightarrow}}

\newcommand{\x}{\ensuremath{\times}}

\renewcommand{\Re}{\ensuremath{\mathrm{Re} \ }}

\newcommand{\dell}{\ensuremath{\partial}}
\newcommand{\dellbar}{\ensuremath{\overline{\partial}}}

\DeclareMathOperator{\Cov}{Cov}

\DeclareMathOperator{\Var}{Var}

\DeclareMathOperator{\Div}{Div}
\begin{document}
\author{Boris Hanin}
\address{Department of Mathematics, Northwestern University, 2033 Sheridan Rd. Evanston, IL, 60208}
\email{bhanin@math.northwestern.edu}

\title[Critical Points on a Riemann Surface]{Pairing of Zeros and Critical Points for Random Meromorphic Functions on Riemann Surfaces}

\setcounter{section}{-1}
\maketitle
\begin{abstract}
We prove that zeros and critical points of a random polynomial $p_N$ of degree $N$ in one complex variable appear in pairs. More precisely, suppose $p_N$ is conditioned to have $p_N(\xi)=0$ for a fixed $\xi\in \C\backslash \set{0}.$ For $\ep\in \lr{0,\frac{1}{2}}$ we prove that there is a unique critical point in the annulus $\setst{z\in \C}{N^{-1-\ep}<\abs{z-\xi}< N^{-1+\ep}}$ and no critical points closer to $\xi$ with probability at least $1-O(N^{-3/2+3\ep}).$ We also prove an analogous statement in the more general setting of random meromorphic functions on a closed Riemann surface. 
\end{abstract}
%\tableofcontents
\begin{figure}[h]
\centering
\includegraphics[scale=.45]{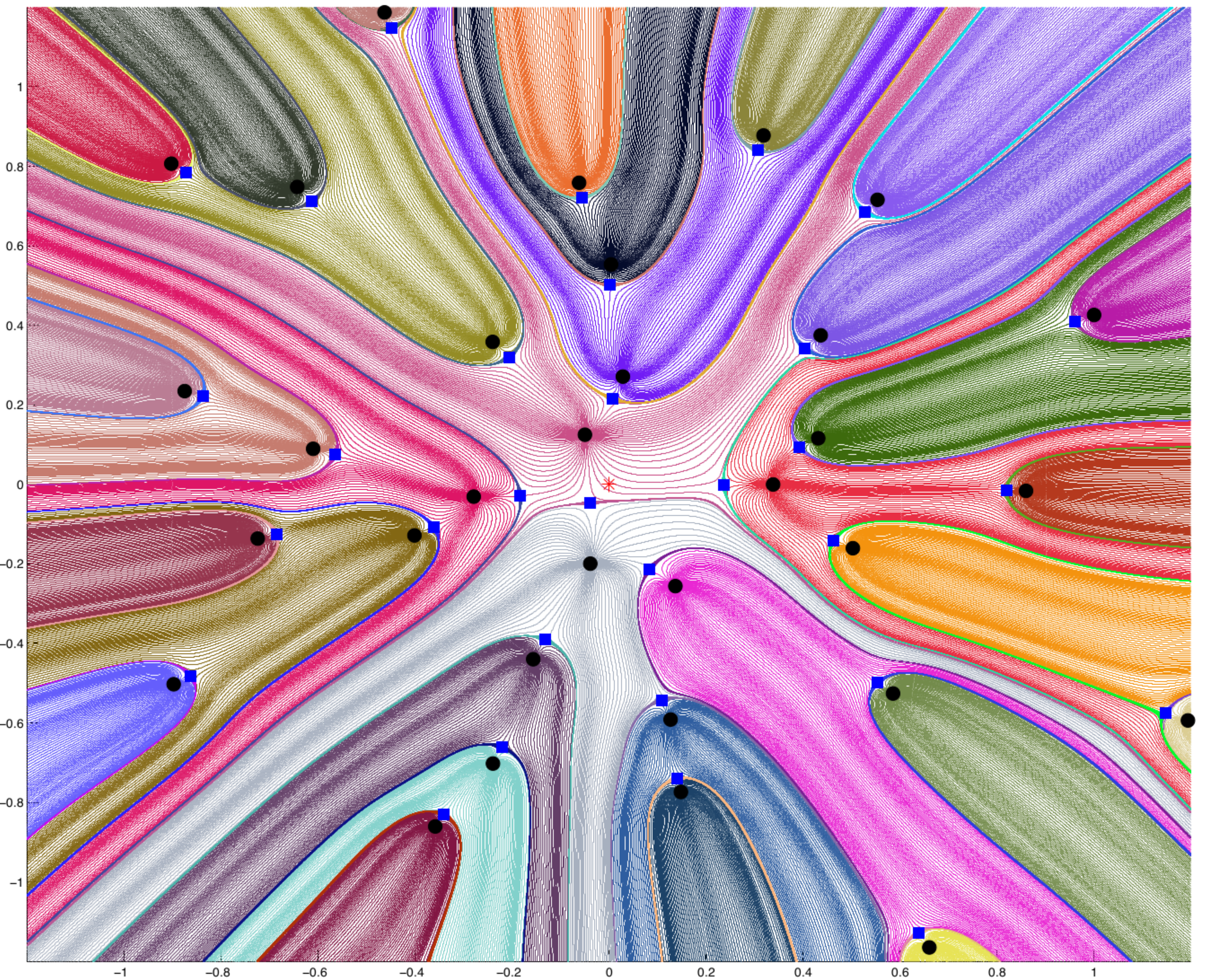}
\label{F:Unscaled SU2 50}
\caption{The zeros (black discs) and critical points (blue squares) of a degree $50$ $SU(2)$ polynomial $p_{50}$ appear in pairs. The typical distance between each pair is on the order of $\frac{1}{50}.$ Each pair lines up with the origin, denoted by a red asterisk. The gradient flow lines for $\log\abs{p_{50}}^2$ are shown.}
\end{figure}

\section{Introduction}
The purpose of this article is to prove that zeros $\set{z_j}$ and critical points $\set{c_j}$ of random meromorphic functions on a Riemann surface come in pairs $(z_j,c_j)$ with $\abs{z_j-c_j}\approx N^{-1},$ where $N$ is the common number of zeros and poles. To explain the result, consider $\curly P_N^{\xi},$ the space of polynomials in one complex variable of degree at most $N$ that vanish at a fixed $\xi \in \C=\C P^1\backslash \set{\infty}.$ We equip $P_N^{\xi}$ with a conditional gaussian measure $\gamma_{h,\xi}^N$ depending on a hermitian metric $h$ on $\mathcal O(1)\surj \C P^1$ (see \S \ref{S:Def} for a precise definition). Consider the random variables
\begin{equation}
\curly N_r:=\#\setst{z\in D_r(\xi)}{\frac{d}{dz}p_N(z)=0},\label{E:Counting Vars}
\end{equation}
where $D_r(\xi)$ is the disk of radius $r,$ and define $e^{-\phi_{z_0}(z)}:=\norm{z_0(z)}_h^2$ for the usual frame $z_0$ of $\mathcal O(1)$ over $\C P^1\backslash \set{\infty}.$
\begin{Thm}\label{T:Polynomial Pairing}
Suppose $d\phi_{z_0}(\xi)\neq 0.$ For $\ep\in (0,\frac{1}{2}),$ write $R_{\pm}=N^{-1\pm \ep}.$ There is $K=K(\ep,h)$ so that
\[\gamma_{h,\xi}^N\lr{\curly N_{R_+}=1\,\,\, \text{and}\,\,\, \curly N_{R_-}=0}\geq 1- K\cdot N^{-3/2+3\ep}.\]
\end{Thm}

\begin{figure}[h]
\centering
\includegraphics[scale=.45]{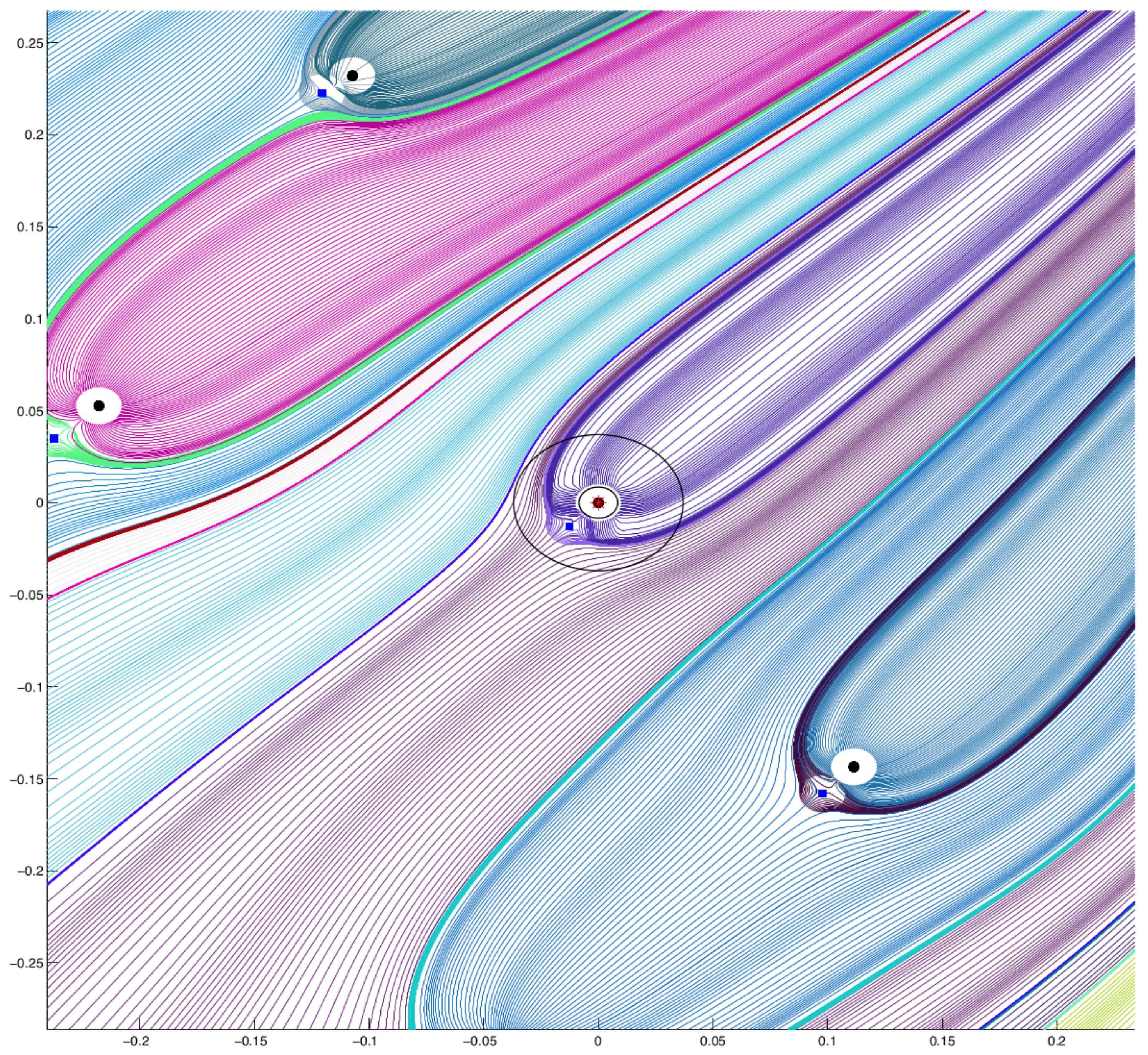}
\label{F:SU2 50 Annulus}
\caption{Zeros (black discs) and holomorphic critical points (blue squares) for an $SU(2)$ polynomial $p_{50}$ of degree $50$ conditioned to have a zero at $\xi=1+1i$ (denoted by a red asterisk) are drawn in normal coordinates centered at $\xi.$ The annulus with inner radius $N^{-1-1/10}$ and outer radius $N^{-1+1/10},$ which Theorem \ref{T:Polynomial Pairing} perdicts to have a unique critical point is shown.}
\end{figure}

\noindent The pairing of zeros and critical points is illustrated in Figures 1 and 2. Typical nearest neighbor distance for $N$ i.i.d points on $\C P^1$ is $N^{-1/2}.$ We give a heuristic derivation of the much smaller $N^{-1}$ distance in Theorem \ref{T:Polynomial Pairing} in terms of electrostatics on a Riemann surface in \S\ref{S:Electrostatics}. In this paper, we focus on understanding the distance from a fixed zero to the nearest critical point for a random polynomial (or more generally meromorphic function on a Riemann surface). In \S \ref{S:Electrostatics} we also give a heuristic explanation for why paired zeros and critical points line up with the origin in Figure 1 and leave a rigorous characterization to future work (cf also Theorem 2 in \cite{ZC}).

\subsection{Riemann Surfaces}
Zeros and critical points of random meromorphic functions on a closed Riemann surface $\Sigma$ also come in pairs. To study this situation we replace $\curly P_N^{\xi}$ with $H_N^{\xi},$ the space of sections of a very ample line bundle $L^{\otimes N}\surj \Sigma$ that vanish at $\xi \in \Sigma.$ We generalize $\frac{d}{dz}$ by fixing an arbitrary section $\sigma_N\in H_N=H_{hol}^0(\Sigma, L^N)$ and defining the meromorphic connection $\nabla^{\sigma_N}$ on $L^N$ via $\nabla^{\sigma_N}\sigma_N=0.$ The critical points of a section $s_N\in H_N^{\xi}$ are thus given by 
\begin{equation}
\nabla^{\sigma_N}s_N(z)=0.\label{E:Crit Equation}
\end{equation}
\noindent The derivative $\frac{d}{dz}=\nabla^{z_0^{\otimes N}}$ on $\mathcal O(N)$ is of this form. For more on meromorphic connections see \S \ref{S:Connection}. Defining the meromorphic function $\gamma_N(z):=\frac{s_N(z)}{\sigma_N(z)}$ on $\Sigma,$ we see that 
\[\nabla^{\sigma_N}s_N(z)=0\quad \longleftrightarrow \quad d\gamma_N(z)=0.\]
$d\gamma_N(z)=0$ is equivalent to $d\log \abs{\gamma_N(z)}=0$ if $\gamma_N$ has simple zeros. We may therefore view $\abs{\gamma_N}$ as a Coulomb potential on $\Sigma$ and interpret the critical point equation $d\log \abs{\gamma_N(z)}=0$ as points of equilibrium for the electric field on $\Sigma$ generated by charge distributed according to the divisor of $\gamma_N.$ This perspective is developed in \S \ref{S:Electrostatics}. 

We emphasize that our notion of critical point is purlely holomorphic and results in a completely different theory from critical points computed with respect to smooth metric connections studied in \cite{VacuaI, VacuaII, VacuaIII, Renjie, CritValues}. We refer the reader to \S \ref{S:Two Connections} for a discussion of this point.

\subsection{Definition of Hermitian Guassian Ensembles}\label{S:Def}
The ensembles of random sections we study are called Hermitian Gaussian Ensembles. They were first studied by Bleher, Shiffman and Zelditch in \cite{PLL, Universality, Quantum, OffDiag}. Let $h$ be a smooth positive Hermitian metric on an ample holomorphic line bundle $L\surj \Sigma$ over a closed Riemann surface. We recall the definition of the Hermitian Gaussian ensemble associated to $h.$ A random section of $L^N$ from this ensemble is
\[s_N:=\sum_{j=0}^N a_j S_j,\]
where $a_j\sim N(0,1)_{\C}$ are i.i.d. standard complex Gaussians and $\{S_j\}_{j=0}^N$ is any orthonormal basis for $H_N$ with respect to the inner product
 \begin{equation}
  \label{E:Inner Product}
\inprod{s_1}{s_2}_h:=\int_{\Sigma}h^N(s_1(z),s_2(z))\, \w_h(z),\quad s_1,s_2\in H_N.
\end{equation}
Here $\w_h:=\frac{i}{2\pi}\partial \cl{\partial} \log h^{-2}$ is the curvature of $(L,h).$ We will write $\gamma_h^N$ for the law of $s_N$ and abbreviate $s_N\in HGE_N(L,h).$ In this paper, we focus on the following variant of $HGE_N(L,h):$ 
\begin{definition}\label{D:Conditional Ensemble}
For $\xi \in \Sigma$ fixed, we will write $s_N\in HGE_N^{\xi}(L,h)$ if the law of $s_N$ is the standard gaussian measure on 
\[H_N^{\xi}=\setst{s_N\in H_{hol}^0(\Sigma, L^N)}{s_N(\xi)=0}\]
generated by the restriction of the inner product (\ref{E:Inner Product}) to $H_N^{\xi}.$ 
\end{definition}
\noindent We denote by $\gamma_{h,\xi}^N$ the law of $s_N\in HGE_N^{\xi}(L,h)$ and will write $s_N(z)=\sum_{j=1}^{d_N-1} a_j S_j^{\xi}(z),$ where $a_j$ are iid standard complex guassians and $\set{S_j^{\xi},\,\, j=1,\ldots, d_N-1}$ is any orthonormal basis for $H_N^{\xi}$ with respect to the inner product (\ref{E:Inner Product}). Every $s_N\in HGE_N^{\xi}(L,h)$ satisfies $s_N(\xi)=0$ and may equivalently be defined as 
\[s_N=\E{\twiddle{s}_N\,|\, ev_{\xi}(s_N)=0},\]
where $\twiddle{s}_N\in HGE_N(L,h)$ and $ev_{\xi}:H_{hol}^0(\Sigma, L^N)\gives L^N|_{\xi}$ is the evaluation map at $\xi.$ See \S 3 of \cite{Conditional} for more details.

\section{Main Result}
\noindent Theorem \ref{T:Pairing} is our main result. We will need the following definition
\begin{definition}\label{D:Preferred Potential}
For $(L,h)\surj \Sigma$ as above and $\sigma_N\in H_{hol}^0(L^N),$ define $\phi_{\sigma_N}:\Sigma\gives (-\infty, \infty]$ by
\[\phi_{\sigma_N}:=\log \norm{\sigma_N}_{h^N}.\]
\end{definition}
\noindent Let $(L,h)\surj \Sigma$ be as above and fix $\Gamma\in [0,\frac{1}{2})$ as well as $\sigma_N\in H_{hol}^0(L^N)$ and $\xi\in \Sigma\backslash\set{\sigma_N =0}$ satisfying 
\begin{equation}
\abs{d\phi_{\sigma^N}(\xi)}> C\cdot N^{1-\Gamma}\label{E:Efield Condition}
\end{equation}
for some $C>0$ and all $N$ (cf \S \ref{S:Kahler Potential} for the electorstatic interpretation of (\ref{E:Efield Condition})).
\begin{Thm}\label{T:Pairing}
Suppose $s_N\in HGE_N^{\xi}(L,h).$ For each $\ep\in (0,\frac{1}{2}-\Gamma)$ define $R_{\pm}=N^{-1+\Gamma \pm \ep}.$ Then there exists a constant $K=K(\Sigma, L, h, \Gamma, \ep),$ such that for each $N\geq 1$ 
\begin{equation}
\gamma_{h,\xi}^N\lr{\curly N_{R_-}=0\quad \text{and}\quad \curly N_{R_+}=1}\geq 1- K\cdot N^{-3/2+2\Gamma+3\ep}.\label{E:FIxed N Pairing}
\end{equation}
 In the definition of $\curly N_r,$ the disk $D_r$ is computed in K\"ahler normal coordinates centered at $\xi.$
\end{Thm}
\begin{remark}
  Let $\mu$ be any probability measure on $\prod_{N=1}^{\infty} H_N^{\xi}$ with marginals $\mu_{h,\xi}^N.$ Write $A_{N,\ep}$ for the event $\left\{\curly N_{R_+}=1\,\,\, \text{and}\,\,\, \curly N_{R_-}=0\right\}.$ If $2\Gamma+3\ep< \frac{1}{2},$ then we may apply the Borel-Cantelli Lemma to see that the events $A_{N,\ep}$ occur for all large enough $N$ $\mu-$almost surely. 
\end{remark}

\section{Discussion}
To explain the pairing of zeros and critical points, let us consider a degree $N$ random polynomial drawn from the $SU(2)$ ensemble studied in \cite{ZC, NV, Conditional, SodTiI, SodTiIII, SodTiII}:
\[p_N(z):=\sum_{j=0}^Na_j \sqrt{\binom{N}{j}}\, z^j.\]
Here $a_j$ are iid standard complex gaussian random variables. The law of $p_N$ is $\gamma_{h_{FS}}^N,$ where $h_{FS}$ is the Fubini-Study metric on $\mathcal O(1)\surj \C P^1.$ $\gamma_{h_{FS}}^N$ is the unique centered gaussian on $\curly P_N$ for which the expected empirical measure of zeros is uniform on $\C P^1$ (cf \S1.2 in \cite{SodinZeros}). The zeros and critical points of $p_{50}$ are drawn in Figure 1. The colored lines are gradient flow lines for the random morse function $G_{50}(z)=\log\abs{p_{50}(z)},$ whose local minima and saddle points are the zeros and critical points of $p_{50},$ respectively. There are no local maxima since $G_{50}$ is subharmonic. Flow lines of the same color terminate in the same zero or critical point.

\subsection{Electrostatic Explanation for Pairing of Zeros and Critical Points}\label{S:Electrostatic Interp}  We now explain why most zeros $z_j$ are paired with unique nearby critical points $c_j$. We also explain why $\arg z_j \approx \arg c_j$ and $0<\abs{z_j}-\abs{c_j}\ll1.$ In fact Theorem \ref{T:Polynomial Pairing} shows that $\abs{z_j}-\abs{c_j}\approx N^{-1}.$ 

Let us distribute $N$ positive and $N$ negative charges on $\C P^1$ according to the divisor of $p_N.$ That is, we place $N$ positive delta charges at infinity and a single negative delta charge at each zero of $p_N.$ Write $E_{p_N}(z)\in T_z^*\C P^1$ for the resulting electric field at $z.$ As explained in \S\ref{S:Electrostatics}, the critical point equation $\frac{d}{dz}p_N(z)=0,$ is equivalent to $E_{p_N}(z)=0.$ 

Suppose that $p_N(\xi)=0$ for some $\xi\neq 0.$ The remaining zeros of $p_N$ tend to be uniformly distributed on $\C P^1.$ For $z$ very near $\xi,$ the contribution to $E_{p_N}(z)$ from the remaining zeros is, heuristically, on the order of $N^{1/2}$ by the central limit theorem. To leading order in $N,$ $E_{p_N}(z)$ is thus the deterministic order of $N$ contribution from the $N$ positive delta charges at infinity and the single negative delta charge at $\xi.$ 

The Coulomb force in $1$ complex dimension at distance $r$ decays like $r^{-1}.$ Hence, for a configuration of $N$ positive charges at infinity and one negative charge at $\xi,$ a point of equilibrium for the electric field exists at a point $z$ a distance of order $N^{-1}$ away from $\xi$ in the direction of the line from infinity to $\xi.$ This is the electrostatic explanation for the pairing of zeros and critical points shown in Figures 1 and 2. 

The pairing of zeros and critical points breaks down near the origin (the south pole) in Figure 1 because the electric field from the $N$ positive charges at infinity vanishes at the south pole. Critical points near $\xi=0$ are therefore determined by the locations of zeros with small modulus. 

\subsection{Electrostatics on Riemann Surfaces}\label{S:Electrostatics}
We describe a theory of electrostatics on a closed Riemann surface $\Sigma$ that depends only on its complex structure. We will see that solutions to the critical point equation $d\gamma=0$ for a meromorphic function $\gamma:\Sigma \gives \C P^1$ are precisely points of equilibrium for the electric field on $\Sigma$ from charges distributed according to its divisor 
\[D:=\Div(\gamma)=\sum_{\gamma(z)=0} m(z) \delta_z-\sum_{\gamma(w)=\infty} m(w) \delta_{w}.\]
Here $m(\cdot)$ denotes the order of the relevant zero or pole of $\gamma.$ To begin, observe that $d\gamma=0$ is equivalent to $d\log\abs{\gamma}=0$ as long as $\gamma$ has simple zeros. Let $\Delta=\frac{i}{\pi}\dell\dellbar$ be the Laplacian mapping $\W^0(\Sigma)$ to $\W^{1,1}(\Sigma).$ 
By the Poincar\'e-Lelong formula, $G(z,D):=\log \abs{\gamma(z)},$ solves
\begin{equation}
\Delta G(z,D)=D.\label{E:Laplace Law}
\end{equation}
This is the analog of Poisson's Equation, which says that the Laplacian of the Coulomb potential gives the charge density. 
\begin{definition}
  The \textit{electric co-field at $z$ from charge distribution $D$} is
  \[E_{\gamma}(z):=d G(z,D_{\gamma})\in T_z^*\Sigma.\] 
\end{definition}
Since $\Sigma$ is compact, the equation $\Delta G =f$ has a solution only if $\int_{\Sigma} f =0.$ The price we pay for using only the complex structure of $\Sigma$ to define $E_{\gamma}$ is that we may work only with electrically netural charge distributions. As noted before, the critical point equation $d\gamma(z)=0$ is generically equivalent to $d\log\abs{\gamma(z)}=0$ and hence to $E_{\gamma}(z)=0.$ 

\subsection{Meaning of $\phi_{\sigma_N}$}\label{S:Kahler Potential}
The quantity $d\phi_{\sigma_N}$ (Definition \ref{D:Preferred Potential}) plays a key role in our results. To see why, note that
\[\frac{i}{2\pi}\dell\dellbar \log \norm{\sigma_N(z)}_h^{-2}=N\w_h - Z_{\sigma_N}.\]
As in \S \ref{S:Def}, $\w_h$ is the curvature form of $h$ and $Z_{\sigma_N}$ is the current of integration over the zero set of $\sigma_N.$ The term $N\w_h$ is essentially $\E{Z_{s_N}-\delta_{\xi}}$ for $s_N\in HGE_N^{\xi}(L,h)$ (cf e.g. Theorem 1 in \cite{Conditional}, Lemma 3.1 in \cite{Quantum}, and Lemma 2 in \S 5 of \cite{ZC}). Let us write as in \S \ref{S:Electrostatic Interp} $E_{\frac{s_N}{\sigma_N}}(\xi)$ for the electric field at $\xi$ from charge distributed according to $Z_{s_N}-\delta_{\xi}-Z_{\sigma_N}.$ Then 
\[d\phi_{\sigma_N}(\xi)\approx \E{E_N(\xi)}.\]
The contribution of the random zeros of $s_N$ should heuristically be on the order of $N^{1/2}$ by the central limit theorem. The condition that $\abs{d\phi_{\sigma_N}(\xi)}>N^{1-\Gamma}$ for some $\Gamma<\frac{1}{2}$ is equivalent to asking that the average electric field at $\xi$ be dominated by the deterministic contribution from the charges at $\xi$ and at $Z_{\sigma_N}.$ Points $z\in \Sigma$ for which $d\phi_{\sigma_N}(z)=0,$ for example, play the same role as the origin for the $SU(2)$ ensemble (cf the end of \S\ref{S:Electrostatic Interp}).

\subsection{Smooth Versus Holomorphic Critical Points}\label{S:Two Connections}
The critical points we study solve the equation
\begin{equation}
\nabla^{\sigma_N}s_N(z)=0\quad \longleftrightarrow \quad \frac{d}{dz}\lr{\frac{s_N(z)}{\sigma_N(z)}}=0.\label{E:Hol Crits}
\end{equation}
Smooth critical points (cf e.g. \cite{VacuaI, VacuaII, VacuaIII}), in contrast, are solutions of
\begin{equation}
\nabla^hs_N(z)=0 \quad \longleftrightarrow \quad \frac{d}{dz} \norm{s_N(z)}_h=0,\label{E:Smooth Crits}
\end{equation}
where $\nabla^h$ is the Chern connection of $h.$ The two settings are qualitatively different. For instance, the zeros of $s_N\in HGE_N(L,h)$ repel (cf e.g. the Introductions in \cite{Universality} and\cite{Conditional}). Hence, since zeros and holomorphic critical points tend to apear in pairs, solutions to (\ref{E:Hol Crits}) repel as well. This can be seen directly by computing the two point function for holomorphic critical points, although we do not do this in the present paper. In contrast, Baber in \cite{Baber} showed that smooth critical points of $s_N$ actually attract. Further, the number of holomorphic critical points depends only on $L,$ $N$, and $\sigma_N$ by the Riemann-Roch formula. The number of smooth critical points is, on the other hand, a non-degenerate random variable, whose expected value is $\frac{5N}{3}$ to leading order in $N$ (cf Corollary 5 and \S 6 in \cite{VacuaI}). 

Smooth critical points were implicitly studied in the work of Nazarov, Sodin, and Volberg \cite{Transportation}, where a so-called ``gravitational allocation'' was constructed between the counting measure for zeros of a gaussian analytic function $f(z)$ and Lebesgue measure on $\C.$ The allocation is achieved by gradient flow for the potential $\norm{f(z)}_h^2,$ where $h(z)=\frac{1}{\pi}e^{-\abs{z}^2}$ is the usual hermitian metric on $\C\x \C\surj \C .$ The saddle points for this potential are critical points of $f$ with respect to the Chern connection of $h.$ The analogous gravitational allocation in Figures 1 and 2 uses $\abs{p_N(z)}^2$ as a potential, omitting the smooth metric factor. Finally, we mention that the expected distribution of critical \textit{values} for smooth critical points was worked out in \cite{Renjie} and \cite{CritValues}.

\section{Acknowledgements}
I am grateful to S. Zelditch for many helpful conversations about his work previous work on statistics of zeros and critical points and for his comments on an earlier draft of this paper. I would also like to thank R. Peled for sharing with me M. Krishnapur's code that I modified to generate the figures in this paper. 

\section{Outline}
The remainder of this paper is organized as follows. First, in \S\ref{S:Connection}, we give some background on meromorphic connections. Then, in \S\ref{S:Notation}, we establish notation to be used throughout. In \S\ref{S:BS} we recall relevant facts about Bergman kernels. Namely, in \S \ref{S:Def SZ}-\ref{S:Normalized Szego} we recall their definition and in \S \ref{S:Principal Bundle}-\ref{S:BS Asymptotics} we recall their asymptotic expansions as given by Zelditch and Shiffman in \cite{OffDiag}. Finally, in the appendix, \S\ref{S:BS Derivs Asymptotics}, we derive asymptotics for derivatives of the Bergman kernel. These asymptotics will be the key analytic formulas underlying the proof of Theorem \ref{T:Pairing}, which is given in \S\ref{S:Pairing Proof}.

\section{Meromorphic Connections on $L\surj \Sigma$}\label{S:Connection}
\begin{definition}\label{D:Meromorphic Connection}
A meromorphic connection on $L\surj \Sigma$ is a connection $\nabla$ on $L$ with the following mapping property:
\[\nabla: H_{mer}^0(\Sigma) \gives H_{mer}^0(\Sigma)\otimes_{\mathcal O_{mer}(\Sigma)}\W^{1,0}(\Sigma).\]
\end{definition}

We study critical points of random sections of $L$ and its tensor powers with respect to a special class of meromorphic connections. For $\sigma\in H_{hol}^0(\Sigma, L)$ the equation $\nabla^{\sigma}\sigma =0$ defines the meromorphic connection $\nabla^{\sigma}$ up a constant multiple. For $s\in H_{hol}^0(L)$ 
\[\nabla^{\sigma}s = \nabla^{\sigma}\frac{s}{\sigma}\sigma = d\lr{\frac{s}{\sigma}}\cdot \sigma \in H_{mer}^0(L)\otimes \W^1(\Sigma).\]
This formula shows that $\nabla^{\sigma}$ introduces a pole at each zero of $\sigma$ (not counting multiplicity). 

Meromorphic connections $L$ are natural generalizations of the holomorphic derivative $\frac{d}{dz}.$ Indeed, if we write $z_0$ for the usual frame for $\mathcal O(1)\surj \C P^1\backslash\set{[0:1]}.$ The section $z_0^N$ induces the trivialization 
\[\alpha_N: \mathcal O(N)\big|_{\C P^1\backslash{\set{\infty}}}\stackrel{\cong}{\longrightarrow} \C \x \C,\]
which identifies $H_{hol}^0(\C P^1,\mathcal O(N))$ with $\curly P_N,$ the polynomials of degree up to $N.$ We may define then define a meromorphic connection $\nabla^{z_0^N}:=\alpha_N^* d$ on $\mathcal O(N).$ We note that the section $z_0^N$ corresponds to the constant polynomial $1$ and hence is parallel for $\nabla^{z_0^N},$ in agreement with our earlier notation. Moreover, $z_0$ vanishes only at $[0:1]$ and hence $\nabla^{z_0^N}$ has a simple pole at $[0:1].$ See \S 3 in \cite{ZC} for more details. 

\section{Notation} \label{S:Notation}
\noindent In \S \ref{S:Def}, we wrote $s_N=\sum_{j=1}^{d_N}a_j S_j\in HGE_N(L,h).$ Consider $\Phi_N:\Sigma \gives \C P^{d_N-1}$ defined by
\[\Phi_N(z)=[S_1(z):\cdots :S_{d_N}(z)].\]
We will refer to $\Phi_N$ as the coherent states embedding generated by $h.$ Since $L$ is ample, the space $H_{hol}^0(\Sigma, L^N)$ is basepoint free for $N$ large and hence $\Phi_N$ is well-defined. The map $\Phi_N$ is an almost-isometry (cf \cite{Tian}): 
\[\norm{\w_h-\frac{1}{N}\Phi_N^* \w_{FS}}_{C^{\infty}}=O(N^{-1}).\]
Here $\w_h$ is the curvature form of $h$ and $\w_{FS}$ is the Fubini-Study metric on $\C P^{d_N}.$ This result is Corollary $3$ in \cite{SzegoTian} and was proved independently by Catlin in \cite{Catlin}. We will write
\[s_N=\inprod{a}{\Phi_N}\]
for $a=(a_1,\ldots, a_{d_N})$ a standard complex gaussian vector on $H_{hol}^0(\Sigma, L^N).$ We assume fixed throughout a distinguished section $\sigma_N\in H_{hol}^0(\Sigma, L^N),$ which is parallel for the meromorphic connection $\nabla^{\sigma_N}$ with respect to which we compute critical points. We will write 
\[S_j=f_j\cdot \sigma_N\]
whenever we do local computations. Theorem \ref{T:Pairing} concerns sections $s_N(z)=\sum_{j=1}^{d_N-1}a_j S_j^{\xi}(z)\in HGE_N^{\xi}(L,h),$ where, as before, $\set{S_j^{\xi}}$ is an orthonomal basis for $H_N^{\xi}$ (defined in the Introduction) with respect to the inner product (\ref{E:Inner Product}). As in the unconditional ensemble, we will write
\[s_N(z)=\inprod{a}{\Phi_N^{\xi}(z)},\]
where $\Phi_N^{\xi}:\Sigma \gives \C P^{d_N-2}$ is given in homogeneous coordinates by $\Phi_N^{\xi}(z)=[S_1^{\xi}(z):\cdots : S_{d_N-1}^{\xi}(z)]$ and set $S_N^{\xi}=f_j^{\xi}\cdot \sigma_N.$ Abusing notation, $\Phi_N^{\xi}$ will sometimes denote a map $\Phi_N^{\xi}:\Sigma \gives \C^{d_N-1}$ given by $\Phi_N^{\xi}(z)=\lr{f_1^{\xi}(z),\ldots, f_{d_N-1}^{\xi}(z)}.$ We will refer to $\Phi_N^{\xi}$ as the conditional coherent states embedding generated by $h$ relative to the frame $\sigma_N.$

\section{Bergman and Szeg\"o Kernels}\label{S:BS}
\noindent Let us recall some background on Bergman kernels. In \S\ref{S:Def SZ}, we define the $N^{th}$ Bergman kernel $\Pi_N$ of $(L,h).$ The related conditional Bergman kernel $\Pi_N^{\xi}$ is the covariance kernel of the Gaussian field $s_N\in HGE_N^{\xi}(L,h).$ In \S\ref{S:Normalized Szego}, we introduce the conditional normalized Bergman kernel $P_N^{\xi},$ key in the proof of Theorem \ref{T:Pairing}. Our main technical tool is the $C^{\infty}$ asymptotic expansion for $\Pi_N$ derived by Shiffman and Zelditch in \cite{OffDiag, NV}, which we recall in \S\ref{S:BS Asymptotics}. Off-diagonal Bergman kernel asymptotic expansions are given also in \cite{HolMorse} and in \cite{DiracBergman}. To explain this asymptotic expansion we recall in \S\ref{S:Principal Bundle} the principal $S^1$ bundle $X\twoheadrightarrow \Sigma$ associated to $(L,h)\surj \Sigma.$ The family of kernels $\Pi_N$ are analyzed by lifting to $X,$ where they are naturally interpreted as Sz\"ego kernels. In the appendix, \S\ref{S:BS Derivs Asymptotics}, we derive asymptotic expansions for derivatives of $\Pi_N^{\xi}$ with respect to the meromorphic connection $\nabla^{\sigma_N}.$ 

\subsection{Definition of $\Pi_N$ and $\Pi_{\sigma_N}$}\label{S:Def SZ} We make the following
\begin{definition}\label{D:Szego}
  The covariance kernel for $s_N=\sum_{j=1}^{d_N} a_j S_j\in HGE_N(L,h)$ is called the $N^{th}$ Bergman kernel for $(L, h):$
  \begin{equation}
    \label{E:Szego Def}
    \Pi_N(z,w):=\Cov(p_N(z),p_N(w))=\sum_{j=0}^N S_N^j(z)\otimes \cl{S_N^j(w)}\in H_{hol}^0(\Sigma,L^N\otimes \cl{L^N}).
  \end{equation}
\end{definition}
The family of Bergman kernels $\Pi_N$ is well-understood for a positive holomorphic line bundle $(L,h)\twoheadrightarrow M$ over a compact K\"ahler manifold $M$ (cf \cite{OffDiag, NV}). If we fix a local frame $e$ for $L^N$ and write $S_j=\gamma_j\cdot e,$ then we can make the following
\begin{definition}\label{D:Bergman}
  The $N^{th}$ Bergman kernel for $(L,h)$ relative to the frame $e$ is 
\[\Pi_e(z,w):=\sum_{j=1}^{d_N}\gamma_j(z)\cl{\gamma_j(w)}.\]
\end{definition}
We observe that $\Pi_e(z,w)\cdot e(z)\otimes \cl{e(w)}=\Pi_N(z,w).$ Writing $s_N(z)=\sum_{j=0}^{d_N-1}a_j S_j^{\xi}(z)\in HGE_N^{\xi}(L,h),$ we define the $N^{th}$ conditional Bergman kernel to be
\[\Pi_N^{\xi}(z,w):=\sum_{j=1}^{d_N-1} S_j^{\xi}(z)\otimes \cl{S_j^{\xi}(w)}.\]
Similarly, writing $S_j^{\xi}=\gamma_j^{\xi} \cdot e$ for any frame $e$ of $L^N,$ we define the $N^{th}$ conditional Bergman kernel relative to the frame $e$ to be
\[\Pi_e^{\xi}:= \sum_{j=1}^{d_N-1} \gamma_j^{\xi}(z)\cl{\gamma_j^{\xi}(w)}\]
and note that
\begin{equation}
\Pi_N^{\xi}(z,w)=\Pi_e^{\xi}(z,w)\cdot e(z)\otimes \cl{e(w)}.\label{E:BS Rln}
\end{equation}

\subsection{Normalized Bergman Kernel}\label{S:Normalized Szego} The local statistics of the critical points for $s_N\in HGE_N^{\xi}(L,h)$ are conveniently expressed in terms of the normalized Bergman kernel for the conditional ensemble $HGE_N^{\xi}(L,h)$ (cf \cite{ZC, NV, Conditional}):
\begin{equation}
  \label{E:Normalized Szego Def I}
P_N^{\xi}(z,w):=\frac{\norm{\nabla^{\sigma_N}\otimes \cl{\nabla^{\sigma_N}} \Pi_N^{\xi}(z,w)}_{h^N}}{\sqrt{\norm{\nabla^{\sigma_N}\otimes \cl{\nabla^{\sigma_N}}\Pi_N^{\xi}(z,z)}_{h^N}\norm{\nabla^{\sigma_N}\otimes  \cl{\nabla^{\sigma_N}}\Pi_N^{\xi}(w,w)}_{h^N}}}.
\end{equation}
The analysis of $P_N^{\xi}$ will play a crucial role in the proof of Theorem \ref{T:Pairing}. Probabilitistically, $P_N^{\xi}(z,w)$ is the correlation between the random variables $\nabla^{\sigma_N} s_N(z)$ and $\nabla^{\sigma_N}s_N(w)$ for $s_N\in HGE_N^{\xi}(L,h).$

\subsection{Principal $S^1$ Bundle}\label{S:Principal Bundle}
Consider a positive line bundle $(L,h)\surj M$ over a compact K\"aher manifold and an orthonormal basis $\set{S_j}_{j=0}^{d_N}$ for $H_{hol}^0(L^N)$ with respect to the inner product (\ref{E:Inner Product}). The $N^{th}$ Bergman kernel $\Pi_N(z,w)=\sum_{j=0}^{d_N} S_j(z)\otimes \cl{S_j(w)}$ is studied in \cite{OffDiag} by lifting sections $s\in H_{hol}^0(L^{\otimes N})$ to $S^1$-equivariant functions on the principal $S^1$ bundle associated to $(L,h),$ where the parametrix construction of Botet de Monvel and Sj\"ostrand in \cite{BMS} may be applied. More precisely, we write $h^*$ for the dual metric on the dual bundle $L^*$ and define the principle $S^1$ bundle $X\surj M$ by
\[X:=\setst{v\in L^*}{\norm{v}_{h^*}=1}.\]
We denote by $\widehat{s}$ the lift of a section $s$ to the function $\widehat{s}(v):=v^{\otimes N}(s)$ on $X.$ Writing $s=f\cdot e^{\otimes N}$ for a local frame $e$ of $L,$ and using $e^*$ to trivialize $X,$ we may write
\begin{equation}
  \label{E:Section Lift}
\widehat{s}(\theta,z):=e^{iN\theta}\norm{e(z)}_h^N\cdot f(z).
\end{equation}
Observe that 
\begin{equation}
  \label{E:Section Drop}
  \abs{\widehat{s}(\theta,z)}=\norm{s(z)}_{h^N}.
\end{equation}
The lifted Bergman kernel is then 
\[\widehat{\Pi}_N(\alpha, z; \beta, w):=\sum_{j=0}^N \widehat{S_j}(\alpha, z)\cl{\widehat{S_j}(\beta, w)}.\]
We observe that $\widehat{\Pi}_N$ is the Szeg\"o kernel for the Hardy space of $X.$ See \S1.2 of \cite{OffDiag} for further details. In this paper, we are interested in the special case $M=\Sigma,$ a closed Riemann surface. The following two definitions from \S2.2 of \cite{NV} allow us to formulate the $C^{\infty}$ complete asymptotic expansion for $\widehat{\Pi}_N$ derived there. 
\begin{definition}\label{D:Preferred Frame}
 Fix $\xi\in \Sigma$ and $e$ a frame for $L$ in a neighborhood $U$ containing $\xi.$ The frame $e$ is called a preffered frame for $h$ at $\xi$ if in a K\"ahler normal coordinate $z:U\gives \C$ centered at $\xi,$ we have
\[\norm{e(z)}_h=1-\frac{1}{2}\abs{z}^2+o(\abs{z}^2).\]
\end{definition}
\begin{definition}\label{D:Heisenberg Coords}
Fix $\xi \in \Sigma,$ a K\"ahler normal coordinate $\psi:U\gives \C$ centered at $\xi,$ and a preferred frame $e$ for $h$ at $\xi.$ Denoting by $\pi$ the projection map $\pi: X\surj M,$ a Heisenberg coordinate on $X$ centered at $\xi$ is a coordinate $\rho:S^1\x \C \gives \pi^{-1}(U)$ given by
\begin{equation}
  \label{E:Heisenberg Def}
\rho(\theta, \psi(z))=e^{i\theta}\norm{e(z)}_h e^*(z).
\end{equation}
\end{definition}
A Heisenberg coordinate on $X$ is therefore the choice of a K\"ahler normal coordinate on $\Sigma$ centered at $\xi$ and a trivialization of $X$ by a preferred frame at $\xi.$ The role of Heisenberg coordinates is that in these special local coordinates, the Szeg\"o kernels $\widehat{\Pi}_N$ have a universal scaling limit depending only on $\dim_{\C}M.$ We refer the interested reader to \S1.3.2 of \cite{Universality} for more details. 

\subsection{Asymptotic Expansion for $\Pi_N$}\label{S:BS Asymptotics}
We now recall for the particular case of $L\surj \Sigma$ the on-diagonal, near off-diagonal, and far off-diagonal asymptotics for the Szeg\"o kernels $\Pi_N$ derived in \cite{OffDiag} and \cite{NV} by Shiffman and Zelditch. We note that the on-diagonal asymptotics were obtained also by Catlin in \cite{Catlin} off-diagonal expansions appeared in \cite{DiracBergman, HolMorse}. The following is a special case of Theorem 2.4 from \cite{NV}.
\begin{Thm}\label{T:Szego Asymptotics} Fix Heisenberg coordinates on $X$ around $\xi\in \Sigma.$ Suppose $b>\sqrt{j+2k},~j,k\geq 0:$ \\
{\bf 1. Far Off-Diagonal. } For $d(z,w)>b\left(\frac{\log N}{N}\right)^{1/2}$ and $j\geq 0,$ we have 
  \begin{equation}
    \label{E:Szego Far Off-Diag}
\nabla^j \widehat{\Pi}_N(\alpha, z; \beta, w)=O(N^{-k}),    
  \end{equation}
where $\nabla^j$ denotes the horizontal lift to $X$ of any $j$ mixed derivatives in $z,\cl{z},w,\cl{w}.$\\
{\bf 2. Near Off-Diagonal. } Let $\ep>0.$ In Heisenberg coordinates (see Definition \ref{D:Heisenberg Coords}) centered at $\xi,$ we have for $\abs{z}+\abs{w}<b\left(\frac{\log N}{N}\right)^{1/2}$ 
  \begin{equation}
    \label{E:Szego Near Off-Diag}
\widehat{\Pi}_N\left(\alpha, z;\beta,w\right)=e^{iN(\alpha-\beta)-z\cdot \cl{w}+\frac{1}{2}\left(\abs{z}^2+\abs{w}^2\right)}[1+R_N(z,w)],    
  \end{equation}
where
\begin{equation}\label{E:Remainder Term}
\nabla^j R_N(z,w)=O(N^{-1/2+\ep}),
\end{equation}
and the implied constant in equation (\ref{E:Remainder Term}) is allowed to depend on $\ep.$ The remainder $R_N$ satisfies in addition, for $j=0,1,2$
\begin{align}\label{E:Remainder Estimate Near-Diag}
  \abs{\nabla^j R_N(z,w)}=O(\abs{z-w}^{2-j}N^{-1/2+\ep})
\end{align}
uniformly for $\abs{z}+\abs{w}<\left(\frac{\log N}{N}\right)^{1/2}$ with the implied constants are independent of $N.$
\end{Thm}

\section{Proof of Theorem \ref{T:Pairing}}\label{S:Pairing Proof}
We first recall the notation. Let $\sigma_N\in H_{hol}^0(\Sigma, L^N)$ be fixed, and define
\[\phi_N(z)=\log \norm{\sigma_N}_h^{-2}.\]
Consider $\xi \in \Sigma\backslash \set{\sigma_N=0}$ such that $\abs{d\phi_N(\xi)}>C\cdot N^{1-\Gamma}$ for a fixed $\Gamma \in [0,1/2)$ and a universal constant $C.$ In a K\"ahler normal coordinate around $\xi,$ we wrote 
\[\curly N_r:=\#\setst{w\in D_{r\cdot N^{-1/2}}(\xi)}{\nabla^{\sigma_N} s_N(w)=0},\]
where $D_R(\xi)$ is the disk of radius $R$ in our fixed coordinate system. We fix an $\ep>0$ such that $\Gamma+\ep<1/2$ and abbreviate $R_\pm:= N^{-1/2+\Gamma\pm \ep}.$ The conclusion of Theorem \ref{T:Pairing} follows easily from Lemmas \ref{L:Expectation} and \ref{L:Variance}, which we prove in Sections \ref{S:Expected Value Comp} and \ref{S:Variance Comp}, respectively.

\begin{Lem}\label{L:Expectation}
Fix $\ep>0$ such $\Gamma + \ep <1/2.$ We have 
\begin{equation}
\E{\curly N_{R_+}}=1+O(N^{-\ep}).\label{E:E Upper Bound}
\end{equation}

Similarly, 
\begin{equation}
\E{\curly N_{R_-}}=O(N^{-\ep}).\label{E:E Lower Bound}
\end{equation}
The implied constants in $O(N^{-\ep})$ depend only on $\ep.$
\end{Lem}

\begin{Lem}\label{L:Variance}
For any $\ep>0$, if $r\leq N^{-1/2+\Gamma+\ep},$ then we have
\[\Var[\curly N_r]=O(N^{-3/2+2\Gamma +3\ep}).\]
The implied constant depends only on $\ep.$
\end{Lem}

Indeed, since $N_r$ is an integer valued random variance, by Chebyshev's inequality:
\begin{align*}
  \Prob{\lr{\curly N_{R_+}=1\cap \curly N_{R_-}=0}^c}&\leq \Prob{\curly N_{R_+} \neq 1}+\Prob{\curly N_{R_-}\neq 0}\\
                                                                 &\leq  \Prob{\abs{\curly N_{R_+}-\E{\curly N_{R_+}}}> 1+O(N^{-\ep})}\\
&\,\,+\Prob{\abs{\curly N_{R_-}-\E{\curly N_{R_-}}}> 1+O(N^{-\ep})}\\
                                                                 &\leq \frac{\Var [\curly N_{R^+}]+\Var[\curly N_{R_{\ep}^-}]}{1+O(N^{-\ep})}\\
                                                                 &=O(N^{-3/2+2\Gamma+3\ep}).
\end{align*}
\subsection{Proof of Lemma \ref{L:Expectation}}\label{S:Expected Value Comp}
Since $\sigma_N(\xi)\neq 0,$ we may write as in \S\ref{S:Notation} $s_N=p_N\cdot \sigma_N = \inprod{a}{\Phi_N^{\xi}}\cdot \sigma_N,$ where $a=\lr{a_0,\ldots, a_{d_N}}$ is a standard gaussian vector on $H_N^{\xi}$ and $\Phi_N^{\xi}(z)=\lr{f_1^{\xi}(z),\cdots, f_{d_N}^{\xi}(z)}$ is the coherent states embedding relative to the frame $\sigma_N.$ We will write 
\[\frac{d}{dz}\Phi_N^{\xi}(z)=\lr{\frac{d}{dz}f_1^{\xi}(z),\cdots, \frac{d}{dz}f_{d_N}^{\xi}(z)}.\]
The $N^{th}$ conditional Bergman kernel relative to $\sigma_N$ (introduced in \S\ref{S:Def SZ}) is therefore $\Pi_{\sigma_N}^{\xi}(z,w)=\sum_{j=1}^{d_n}f_j^{\xi}(z)\cl{f_j^{\xi}(w)}.$ Our proof of Lemma \ref{L:Expectation} proceeds in two steps: 
\begin{enumerate}
\item First, in Lemma \ref{L:BK E Formula}, we write, for any $r>0,$ $\E{\curly N_r}$ in terms of $\Pi_{\sigma_N}^{\xi}(z,w)$ by using the Poincar\'e-Lelong formula: 
\[\E{\curly N_r}=\frac{r}{2\pi}\int_0^{2\pi}I(re^{i\theta})d\theta,\]
where
\[I(z):=\frac{d}{dz}\log \lr{\frac{d^2}{dzd\cl{w}}\bigg|_{z=w=re^{i\theta}}\Pi_{\sigma_N}^{\xi}(\widehat{z},\widehat{w})}\cdot e^{i\theta},\]
and $\widehat{u}=u\cdot N^{-1/2}.$
\item We then complete the proof by using the estimate (\ref{E:DLog T}) to see that if $\abs{z}=R_+,$ then $I(z)=\frac{R_+^{-1}\lr{1+O(N^{-\ep})}}{1+O(N^{-\ep})},$ while if $\abs{z}=R_-,$ then $I(z)=\frac{O(R_-^{-1}N^{-\ep})}{1+O(N^{-\ep})}.$ 
\end{enumerate}

\begin{Lem}\label{L:BK E Formula}
  We have
  \begin{equation}
\E{\curly N_r}=\frac{r}{2\pi}\int_0^{2\pi}\frac{d}{dz}\log \lr{\frac{d^2}{dzd\cl{w}}\bigg|_{z=w=re^{i\theta}}\Pi_{\sigma_N}^{\xi}(\widehat{z},\widehat{w})}\cdot e^{i\theta}d\theta.\label{E:Expected Value Bergman Kernel}
\end{equation}
\end{Lem}
\begin{proof}
The Poincar\'e-Lelong formula states the current of integration on the zero set of a non-zero analytic function $f$ is $Z_f=\frac{i}{2\pi}\dell\dellbar \log \abs{f}^2.$ Hence, 
\begin{equation}
\curly N_r=\int_{D_{r\cdot N^{-1/2}}} Z_{p_N}=\int_{D_{r\cdot N^{-1/2}}} \frac{i}{2\pi}\dell_z \dellbar_z \log \abs{\frac{d}{dz}p_N(z)}^2.\label{E:Counting Via PLL}
\end{equation}

Since $\dell\dellbar = -d\dell,$ we may use Stokes theorem to write 
\begin{equation}
\curly N_r=\frac{1}{2\pi i} \oint_{\dell D_{r\cdot N^{-1/2}}}\dell \log \abs{\frac{d}{dz}p_N(z)}^2=\frac{1}{2\pi i}\oint_{\dell D_r} \dell \log \abs{p_N(\widehat{z})}^2.\label{E:Arg Prin}
\end{equation}
Applying Fubini's theorem and differentiating under the integral sign, we find that
\[\E{\curly N_r}=\frac{r}{2\pi}\int_0^{2\pi} \frac{d}{dz}\E{\log \abs{\frac{d}{dz} p_N(re^{i\theta}\cdot N^{-1/2})}^2} \cdot e^{i\theta}d\theta.\]
Writing $\norm{\cdot}$ for the $l_2$ norm a vector, we note that 
\[\E{\log \abs{\frac{d}{dz}p_N(\widehat{z})}}^2=\E{\log \abs{\inprod{a}{\frac{\frac{d}{dz}\Phi_N^{\xi}(\widehat{z})}{\norm{\frac{d}{dz}\Phi_N^{\xi}(\widehat{z})}}}}^2}+\log \norm{\frac{d}{dz}\Phi_N^{\xi}(\widehat{z})}^2.\]
Since the gaussian measure is unitarily invariant, we see that the first term is $\E{\log \abs{a_1}},$ and there is therefore annihilated by $\frac{d}{dz}.$ The following observation completes the proof: 
\[\norm{\frac{d}{dz}\Phi_N^{\xi}(\widehat{z})}^2=\sum_{j=1}^{d_N}\abs{\frac{d}{dz}f_j^{\xi}(\widehat{z})}^2= \frac{d^2}{dzd\cl{w}}\bigg|_{z=w} \Pi_{\sigma_N}^{\xi}(\widehat{z}, \widehat{w}).\]
\vskip -.5cm
\end{proof}
\noindent Combining formula (\ref{E:Expected Value Bergman Kernel}) with (\ref{E:Log BN}), we find
\begin{equation}
\E{\curly N_r}=\frac{r}{2\pi}\int_0^{2\pi} \lr{N^{-1/2} \Di{z}{}\bigg|_{z=re^{i\theta}}\phi_N(\widehat{z})+\frac{d}{dz}\bigg|_{z=re^{i\theta}}\log T(z,z)+O(N^{-1/2+\ep})}e^{i\theta}d\theta.\label{E:E Integral}
\end{equation}
We may taylor expand to write
\[\Di{z}{\phi_N}(\widehat{z})=\Di{z}{\phi_N}(\xi)+O(\abs{z}N^{-1/2}).\]
Hence, 
\[\frac{r}{2\pi}\int_0^{2\pi} N^{-1/2}\Di{z}{}\bigg|_{z=re^{i\theta}} \phi_N (\widehat{z})e^{i\theta}d\theta = O(r^2N^{-1}).\]
Therefore, we find
\[\E{\curly N_r}=\frac{r}{2\pi}\int_0^{2\pi}\frac{d}{dz}\bigg|_{z=re^{i\theta}} \log T(z,z)e^{i\theta}d\theta + O(rN^{-1/2+\ep}).\]
Finally, using the estimate (\ref{E:DLog T}), we have 
\[\frac{d}{dz}\bigg|_{z=R_-\cdot e^{i\theta}} \log T(z,z)= \frac{R_-^{-1}\lr{e^{-i\theta}+O(N^{-1/2+\Gamma})}}{1+N^{2\ep}\lr{1+O(N^{-\ep})}}.\]
Substituting this expression into the integral (\ref{E:E Integral}), we conclude (\ref{E:E Lower Bound}). Similarly, the estimate (\ref{E:DLog T}) yields
\[\frac{d}{dz}\bigg|_{z=R_+\cdot e^{i\theta}} \log T(z,z)= \frac{R_+^{-1}\lr{e^{-i\theta}+O(N^{-1/2+\Gamma})}}{1+O(N^{-\ep})},\]
from which we deduce (\ref{E:E Upper Bound}).

\subsection{Proof of Lemma \ref{L:Variance}}\label{S:Variance Comp}
Note that we may write
\[P_N^{\xi}(z,w)=\frac{\abs{\frac{d^2}{dzd\cl{w}}\Pi_{\sigma_N}^{\xi}(z,w)}}{\sqrt{\frac{d^2}{dzdw}\big|_{w=z}\Pi_{\sigma_N}^{\xi}(z,w)\cdot \frac{d^2}{dzdw}\big|_{z=w}\Pi_{\sigma_N}^{\xi}(z,w)}},\]
where, $\Pi_{\sigma_N}^{\xi}(z,w)=\sum_{j=1}^{d_n}f_j^{\xi}(z)\cl{f_j^{\xi}(w)}$ is the Bergman kernel relative to $\sigma_N$ (cf \S \ref{S:Def SZ}). 
\begin{Lem}\label{L:Var N Formula}
Let us write $\widehat{u}=u\cdot N^{-1/2}.$ We have the following formula for $\Var[\curly N_r]:$
  \begin{equation}
\frac{r^2}{\lr{2\pi}^2}\int_0^{2\pi}\int_0^{2\pi}\frac{d}{dz}\frac{d}{dw}\bigg|_{z=re^{i\theta_1},\, w=re^{i\theta_2}}G(P_N^{\xi}(\widehat{z},\widehat{w}))e^{i\lr{\theta_1+\theta_2}}d\theta_1d\theta_2,\label{E:Var N Formula}
\end{equation}
where 
\[G(t)=\frac{\gamma^2}{4}-\frac{1}{4}\int_0^{t^2}\frac{\log(1-s)}{s}ds.\]
\end{Lem}
\begin{proof}
Using equation (\ref{E:Counting Via PLL}), we have that
\begin{equation}
\E{\curly N_r^2}=\frac{r^2}{\lr{2\pi}^2}\int_0^{2\pi}\int_0^{2\pi}\frac{d^2}{dzdw}\bigg|_{z=w=re^{i\theta}}\E{\log \abs{p_N(\widehat{z})}\log \abs{p_N(\widehat{w})}}e^{i\lr{\theta_1+\theta_2}}d\theta_1d\theta_2.\label{E:NR2}
\end{equation}
\noindent For any vector $v\in \C^{d_N-1}\backslash\set{0},$ we will write $\twiddle{v}=\frac{v}{\norm{v}}.$ We have
\[\log \abs{p_N(\widehat{z})}=\log \abs{\inprod{a}{\Phi_N^{\xi}(\widehat{z})}}=\log\abs{\inprod{a}{\twiddle{\Phi_N^{\xi}}(\widehat{z})}}+\log \norm{\Phi_N^{\xi}(\widehat{z})}\]
and similarly for $\log \abs{p_N(\widehat{w})}.$ We therefore find that 
\[\E{\curly N_r^2}=\frac{r^2}{\lr{2\pi}^2}\int_0^{2\pi}\int_0^{2\pi}\frac{d^2}{dzdw}\bigg|_{z=re^{i\theta_1},\, w=re^{i\theta_2}}\lr{E_1+E_2+E_3+E_4}e^{i\lr{\theta_1+\theta_2}}d\theta_1d\theta_2,\]
where
\begin{align*}
E_1(z,w)&:= \log \norm{\Phi_N^{\xi}(\widehat{z})}\log\norm{\Phi_N^{\xi}(\widehat{w})}\\
E_2(z,w)&:= \E{\log \abs{\inprod{a}{\twiddle{\Phi_N^{\xi}}(\widehat{z})}}}\cdot \log \norm{\Phi_N^{\xi}(\widehat{w})}\\
E_3(z,w)&:= \E{\log \abs{\inprod{a}{\twiddle{\Phi_N^{\xi}}(\widehat{w})}}} \cdot\log\norm{\Phi_N^{\xi}(\widehat{z})}\\
E_4(z,w)&:= \E{\log \abs{\inprod{a}{\twiddle{\Phi_N^{\xi}}(\widehat{z})}} \log \abs{\inprod{a}{\twiddle{\Phi_N^{\xi}}(\widehat{w})}}}.
\end{align*}
Since the gaussian measure $a$ is unitarily invariant, we see that $E_2(z,w),E_3(z,w)$ are independent of $z,w,$ respectively and hence are annihilated by $\frac{d^2}{dzdw}.$ Moreover, 
\begin{equation}
\frac{r^2}{\lr{2\pi}^2}\int_0^{2\pi}\int_0^{2\pi}\frac{d^2}{dzdw}\bigg|_{z=re^{i\theta_1},\, w=re^{i\theta_2}}E_1(z,w) e^{i\lr{\theta_1+\theta_2}}=\E{N_r}^2.\label{E:E1 Integral}
\end{equation}
In order to interpret $E_4(z,w),$ we now recall the following result.
\begin{Lem}[Lemma 3.3 from \cite{NV}]\label{L:G} Let $a$ be a standard Gaussian random vector in $\C^{N+1}$ and let $u,v\in C^{N+1}$ denote unit vectors. Then 
\[\E{\log\abs{\inprod{a}{u}}\log\abs{\inprod{a}{v}}}=G(\abs{\inprod{u}{v}}),\]
where $\inprod{\cdot}{\cdot}$ is the usual Hermitian inner product on $\C^{N+1}.$
\end{Lem}
\noindent Observe that $\abs{\inprod{\twiddle{\Phi_N^{\xi}}(\widehat{z})}{\twiddle{\Phi_N^{\xi}}(\widehat{w})}}=P_N^{\xi}(\widehat{z},\widehat{w}).$ Putting this together with (\ref{E:E1 Integral}), we find that
\[\Var[\curly N_r]=\frac{r^2}{\lr{2\pi}^2}\int_0^{2\pi}\int_0^{2\pi}\frac{d^2}{dzdw}\bigg|_{z=re^{i\theta_1},\, w=re^{i\theta_2}}G(P_N^{\xi}(\widehat{z},\widehat{w})) e^{i\lr{\theta_1+\theta_2}}d\theta_1d\theta_2,\]
as claimed. 
\end{proof}
\noindent To complete the proof, note that
\begin{equation}
G'(t)=-\frac{\log(1-t^2)}{2t}\qquad G''(t)=\frac{1}{1-t^2}+\frac{\log(1-t^2)}{2t^2}.\label{E:F Derivs}
\end{equation}
Hence, we can write $\Var[\curly N_r]$ as 
\begin{equation}
\frac{r^2}{2\pi}\int_0^{2\pi}\int_0^{2\pi} e^{i\lr{\theta_1+\theta_2}}I(z,w)\bigg|_{z=re^{i\theta_1},\, w= re^{i\theta_2}} d\theta_1 d\theta_2,\label{E:Variance Integral}
\end{equation}
where
\[I(z,w):=G''(P_N^{\xi}(\widehat{z},\widehat{w}))\Di{z}{}P_N^{\xi}(\widehat{z},\widehat{w})\cdot \Di{w}{}P_N^{\xi}(\widehat{z},\widehat{w})+G'(P_N^{\xi}(\widehat{z},\widehat{w}))\frac{d^2}{dzdw}P_N^{\xi}(\widehat{z},\widehat{w}).\]
\noindent Substituting (\ref{E:DZ Lambda})-(\ref{E:Den}) of Corollary \ref{C:AA Normalized Szego} into (\ref{E:Variance Integral}) and noting that $\log\abs{re^{i\theta_1}-re^{i\theta_2}}$ has finite integral completes the proof.

\section{Appendix: Asymptotic Expansions for Bergman Kernel Derivatives}\label{S:BS Derivs Asymptotics} 
We now apply the asymptotic expansions for the Bergman kernel from \S\ref{S:BS Asymptotics} to obtain asymptotic expansions for its derivatives given in Lemma \ref{L:AA Cond Bergman Kernel Off-Diag} and Corollaries \ref{C:AA Log Cond Bergman Kernel} and \ref{C:AA Normalized Szego}. These will be the crucial technical tools in proving Theorem \ref{T:Pairing}. 

\begin{Lem}\label{L:AA Cond Bergman Kernel Off-Diag}
Fix $\xi \in \Sigma.$ Write $\Pi_{\sigma_N}^{\xi}$ for the $N^{th}$ conditional Bergman kernel relative to $\sigma_N$ (defined in \S\ref{S:Def SZ}). In K\"ahler normal coordinates around $\xi,$ we write $\widehat{u}=u\cdot N^{-1/2}.$ The following expression is valid uniformly for $\abs{\widehat{z}},\abs{\widehat{w}}<\sqrt{\log N}$
\begin{equation}
\frac{d^2}{dzd\cl{w}} \Pi_{\sigma_N}^{\xi}(\widehat{z},\widehat{w})=\frac{N^3}{\pi}e^{\frac{1}{2}\lr{\phi_N(\widehat{z})+\phi_N(\widehat{w})+i\twiddle{\gamma_N}(\widehat{z})-i\twiddle{\gamma_N}(\widehat{w})+2z\cl{w}-\abs{z}^2-\abs{w}^2}}\cdot T(z,w)\cdot \lr{1+R_N(z,w)},\label{E:BN Off-Diag}
\end{equation}
We've written
\begin{align*} 
T(z,w)&=\lr{1-e^{-z\cl{w}}}\left[1+\frac{1}{4}\lr{\Di{z}{\phi_N}(\widehat{z})N^{-1/2}+\Di{z}{\gamma_N}(\widehat{z})N^{-1/2}-\cl{z}+2\cl{w}}\right.\\
&\left.\cdot \lr{\Di{\cl{w}}{\phi_N}(\widehat{w})N^{-1/2}+\Di{\cl{w}}{\gamma_N}(\widehat{w})N^{-1/2}-w+2z}\right]\\
&+e^{-z\cl{w}}\left[1-z\cl{w}+\frac{z}{2}\lr{\Di{z}{\phi_N}(\widehat{z})N^{-1/2}+\Di{z}{\gamma_N}(\widehat{z})N^{-1/2}-\cl{z}+2\cl{w}}\right.\\
&\left. +\frac{\cl{w}}{2}\lr{\Di{\cl{w}}{\phi_N}(\widehat{w})N^{-1/2}+\Di{\cl{w}}{\gamma_N}(\widehat{w})N^{-1/2}-w+2z}\right]
\end{align*}
We've also set $\gamma_N$ to be the ``leading harmonic part of $\phi_N$'':
\[\gamma_N(z,\cl{z}):=\phi_N(\xi)+\Di{z}{\phi_N}\bigg|_{\xi}\cdot z+\Di{\cl{z}}{\phi_N}\bigg|_{\xi} \cdot \cl{z}+\frac{1}{2}\left[\DDi{z}{\phi_N}\bigg|_{\xi}\cdot z^2+\DDi{\cl{z}}{\phi_N}\bigg|_{\xi}\cdot \cl{z}^2\right],\]
and we've written $\twiddle{\gamma_N}$ for its harmonic conjugate. Finally, as in Theorem \ref{T:Szego Asymptotics}, the remainders $R_N(z,w)$ satisfy the estimates (\ref{E:Remainder Estimate Near-Diag}). 
\end{Lem}
Before proving Lemma \ref{L:AA Cond Bergman Kernel Off-Diag}, we record several corollaries. 

\begin{corollary}\label{C:AA Log Cond Bergman Kernel}
With the notation of Lemma \ref{L:AA Cond Bergman Kernel Off-Diag} and for any $\ep>0$, the following expression is valid uniformly for $\abs{z}<\sqrt{\log N}$
\begin{equation}
\log\left[ \frac{d^2}{dzd\cl{w}}\bigg|_{z=w} \Pi_{\sigma_N}^{\xi}(\widehat{z},\widehat{w})\right]= C + \phi_N(\widehat{z}) + \log T(z,z)+O(N^{-1/2+\ep}),\label{E:Log BN}
\end{equation}
where $C$ is a constant depending on $N$ and, if we write $z=re^{i\theta},$ we have for $r$ small
\begin{equation}
  \label{E:DLog T}
  \frac{d}{dz}\bigg|_{z=re^{i\theta}}\log T(z,z)= \frac{r^{-1}\lr{e^{-i\theta}+O(N^{-1/2+\Gamma})}}{1+C_N\cdot r^{-2}N^{-1+2\Gamma}+O(r^{-1}N^{-1/2+\Gamma})+O(r^{-2}N^{-2+2\Gamma})},
\end{equation}
where $C_N$ are constants bounded away from $0$ and $\infty$ uniformly in $N.$
\end{corollary}
\begin{proof}
  Equation (\ref{E:Log BN}) follows from setting $z=w$ in (\ref{E:BN Off-Diag}). To derive (\ref{E:DLog T}), we put $z=w$ in the expression for $T(z,w)$ given in Lemma \ref{L:AA Cond Bergman Kernel Off-Diag} to see that
\begin{align*} T(z,z)&=\lr{1-e^{-\abs{z}^2}}\lr{1+\frac{1}{4}\abs{\Di{z}{\phi_N}(\widehat{z})N^{-1/2}+\Di{z}{\gamma_N}(\widehat{z})N^{-1/2}+z}^2}\\
&+e^{-\abs{z}^2}\Re\lr{1+\cl{z}\lr{\Di{z}{\phi_N}(\widehat{z})N^{-1/2}+\Di{z}{\gamma_N}(\widehat{z})N^{-1/2}}}.
\end{align*}
Observe that 
\[\frac{1}{2}\lr{\Di{z}{\phi_N}(\widehat{z})N^{-1/2}+\Di{z}{\gamma_N}(\widehat{z})N^{-1/2}+z}=\Di{z}{\phi_N(\xi)}N^{-1/2}+O(N^{-1})\]
Recalling that $\abs{\Di{z}{\phi_N}(\xi)}>C\cdot N^{-1+\Gamma}$ for some $\Gamma\in [0,\frac{1}{2}),$ we may set $z=re^{i\theta}$ and taylor expand to find that 
\[T(z,z)=r^2\cdot N^{-1} \lr{1+\abs{\Di{z}{\phi_N}(\xi)}^2}\lr{1+C_N\cdot N^{-1+2\Gamma}r^{-2}+O(r^{-1}N^{-1/2+\Gamma})}.\]
Similarly, we find that
\[\frac{d}{dz}T(z,z)=r^2\cdot N^{-1}\lr{1+\abs{\Di{z}{\phi_N}(\xi)}^2}r^{-1}\lr{e^{-i\theta}+O(rN^{-1/2+\Gamma})}.\]
Combining the expressions for $T(z,z)$ and $\frac{d}{dz}T(z,z)$ yields (\ref{E:DLog T}).

\end{proof}

Lemma \ref{L:AA Cond Bergman Kernel Off-Diag} allows us to conclude the following asymptotic expansion for $P_N^{\xi}.$ 
\begin{corollary}\label{C:AA Normalized Szego}
 Fix a K\"ahler normal coordinate centered at $\xi,$ and write $\widehat{u}=u\cdot N^{-1/2}.$ Then, for any $\ep>0,$ we have the following $C^{\infty}$ expansion:
 \begin{equation}
P_N^{\xi}(\widehat{z},\widehat{w})=\frac{\abs{1-e^{z\cl{w}}}}{\lr{1-e^{-\abs{z}^2}}^{1/2}\lr{1-e^{-\abs{w}^2}}^{1/2}}\cdot e^{-\frac{1}{2}\abs{z-w}^2}\lr{1+R_N(z,w)}.\label{E:Normalized Conditional}
\end{equation}

The remainder $R_N(z,w)$ satisfies the estimates (\ref{E:Remainder Estimate Near-Diag}). In particular, we find that for $\abs{z},\abs{w}$ small, we have for some constants $C_j,\, j=1,2,$
  \begin{align}
  \label{E:DZ Lambda}  \frac{d}{dz}P_N^{\xi}\lr{\widehat{z},\widehat{w}} &=  \lr{\cl{z}-\cl{w}} \lr{C_1+O(N^{-1/2+\ep})}\\
  \label{E:DW Lambda}  \frac{d}{dw}P_N^{\xi}\lr{\widehat{z},\widehat{w}}&=  \lr{z-w} \lr{C_1+O(N^{-1/2+\ep})}\\
  \label{E:DZDW Lambda}  \frac{d^2}{dzdw}P_N^{\xi}\lr{\widehat{z},\widehat{w}}&= O(N^{-1/2+\ep})\\
  \label{E:Den}  1-e^{-2\Lambda(z,w)}&=1-P_N^{\xi}(z,w)^2=\abs{z-w}^2\lr{C_2+O(\abs{z-w}^2)}.
  \end{align}
Moreover, the constants $C_j$ are uniformly bounded independent of $N.$
\end{corollary}
\begin{proof}
Notice that
\[\frac{d^2}{dzd\cl{w}}B{\sigma_N}^{\xi}(\widehat{z},\widehat{w})=\inprod{\frac{d}{dz}\Phi_N^{\xi}(\widehat{z})}{\frac{d}{dw}\Phi_N^{\xi}(\widehat{w})}.\]
Therefore, by Cauchy-Schwartz, we find that 
\[P_N^{\xi}(\widehat{z},\widehat{w})\leq 1.\]
On the other hand, we see that $P_N^{\xi}(\widehat{z},\widehat{z})=1.$ Therefore, writing
\begin{equation}
P_N^{\xi}(\widehat{z}, \widehat{w})=\frac{\abs{\frac{d^2}{dzd\cl{w}}B_{N, \sigma_N}^{\xi}(\widehat{z},\widehat{w})}}{\lr{\frac{d^2}{dzd\cl{w}}\big|_{z=w}B_{N,\sigma_N}^{\xi}(\widehat{z},\widehat{w})\cdot \frac{d^2}{dzd\cl{w}}\big|_{w=z}B_{N,\sigma_N}^{\xi}(\widehat{z},\widehat{w})}^{1/2}},\label{E:PN From BN}
\end{equation}
we see that the normalized expression
\[\frac{\abs{T(z,w)}}{\sqrt{T(z,z)T(w,w)}}\]
achieve a strict maximum value of $1$ when $z=w,$ and hence we may write
\[\frac{\abs{T(z,w)}}{\sqrt{T(z,z)T(w,w)}}=1+Q_N(z,w),\]
where the remainder terms $Q_N(z,w)$ satisfy the estimates required of the remainders $R_N.$ Substituting expression (\ref{E:BN Off-Diag}) into (\ref{E:PN From BN}) shows that 
\[P_N^{\xi}(\widehat{z}, \widehat{w})= \frac{\abs{1-e^{-z\cl{w}}}e^{-\frac{1}{2}\abs{z-w}^2}}{\lr{1-e^{-\abs{z}^2}}^{1/2}\lr{1-e^{-\abs{w}^2}}^{1/2}}\cdot \lr{1+R_N(z,w)}\]
for with $R_N(z,w)=O(N^{-1/2+\ep}).$ The estimates (\ref{E:Remainder Estimate Near-Diag}) follow from the analogous estimates in Theorem \ref{T:Szego Asymptotics}.

If we fix $w$ and view $P_N^{\xi}(z,w)$ as a function of $z,\cl{z},$ we see that $P_N^{\xi}$ is maximized on the diagonal $z=w$ and achieves a value of $1.$ Estimates (\ref{E:DZ Lambda}) and (\ref{E:DW Lambda}) now follow. To verify (\ref{E:Den}) we may write 
\[P_N^{\xi}\lr{\widehat{z},\widehat{w}}^2=e^{-\Lambda(z,w)},\]
where we may write $\Lambda(z,w):=-\log P_N^{\xi}\lr{\widehat{z},\widehat{w}}$ as
\[\frac{1}{2}\lr{\log(1-e^{-\abs{z}^2})+\log(1-e^{-\abs{w}^2})-\log\abs{1-e^{-z\cl{w}}}+\abs{z-w}^2}+\log\lr{1+R_N(z,w)}.\]
Therefore, 
\[\frac{d^2}{dzdw}P_N^{\xi}\lr{\widehat{z},\widehat{w}}= \lr{-\frac{d^2}{dzdw}\Lambda(z,w)+\Di{z}{}\Lambda(z,w)\Di{w}{}\Lambda(z,w)}e^{-\Lambda(z,w)}.\]
Differentiating $\Lambda(z,w)$ and using the remainder estimates (\ref{E:Remainder Estimate Near-Diag}) completes the derivation. 
\end{proof}
We now turn to the proof of Lemma \ref{L:AA Cond Bergman Kernel Off-Diag}.

\begin{proof}[Proof of Lemma \ref{L:AA Cond Bergman Kernel Off-Diag}]
We fix K\"ahler normal coordinates around $\xi.$ Our proof is based on Lemma \ref{L:Step 1}. To formulate it, we continue to write $\phi_N$ for the K\"ahler potential for $\w_h$ relative to $\sigma_N:$
\[\phi_{N}(z)=\log \norm{\sigma_N(z)}_h^{-2}\]
valid near $\xi.$ We will also write $\twiddle{\gamma_N}$ for the harmonic conjugate of $\gamma_N.$

\begin{Lem}\label{L:Step 1}
For each $N,$ we have
\[\Pi_{\sigma_N}^{\xi}(z,w)=E_N(z,\alpha, w,\beta)\cdot\lr{\widehat{\Pi}_N(z,\alpha, w,\beta)-\frac{\widehat{\Pi}_N(z,\alpha, \xi, 0)\cl{\widehat{\Pi}_N(\xi,0,w,\beta)}}{\abs{\widehat{\Pi}_N(\xi,\alpha, \xi,\beta)}}},\]
where
\[E_N(z,\alpha, w,\beta)=e^{\frac{N}{2}\lr{\phi(z)+\phi(w)-i\twiddle{\gamma_N}(w)+i\twiddle{\gamma_N}(z)}+iN\lr{\alpha- \beta}}.\] 
\end{Lem}

Assuming this Lemma for the moment, we substitute into (\ref{E:Conditioned Szego Lift}) the $C^{\infty}$ asymptotic expansion 
\[\widehat{\Pi}_N(\widehat{z},\alpha, \widehat{w}, \beta)= e^{-z\cl{w}+\frac{1}{2}\lr{\abs{z}^2+\abs{w}^2}}+O(N^{-1/2+\ep}),\]
from (\ref{E:Szego Near Off-Diag}), which is valid in Heisenberg coordinates on $\xi,$ to obtain the following expression for $\Pi_{\sigma_N}^{\xi}(\widehat{z},\widehat{w}):$
\begin{equation}
\frac{N}{\pi}e^{\frac{1}{2}\lr{\phi_N(\widehat{z})+\phi_N(\widehat{w})+i\twiddle{\gamma_N}(\widehat{z})-i\twiddle{\gamma_N}(\widehat{w})+2z\cl{w}-\abs{z}^2-\abs{w}^2}}\lr{1-e^{-z\cl{w}}}\lr{1+R_N(z,w)}.\label{E:Conditional Bergman Scaled}
\end{equation}
Differentiating this expression in $z$ and in $\cl{w}$ shows that 
\[\frac{d^2}{dzd\cl{w}}\Pi_{\sigma_N}^{\xi}(\widehat{z},\widehat{w})=\frac{N}{\pi} e^{\frac{1}{2}\lr{\phi_N(\widehat{z})+\phi_N(\widehat{w})+i\twiddle{\gamma_N}(\widehat{z}-i\twiddle{\gamma_N}(\widehat{w}))}} \cdot T(z,w)\cdot \lr{1+R_N(z,w)},\]
as desired. We now turn to the proof of Lemma \ref{L:Step 1}.

\begin{proof}[Proof of Lemma \ref{L:Step 1}]
From (\ref{E:BS Rln}), we see immediately that
\begin{equation}
\widehat{\Pi}_N^{\xi}(z,\alpha, w,\beta)= \Pi_{\sigma_N}^{\xi}(z,w)\cdot \widehat{\sigma}_N(z,\alpha)\otimes \widehat{\cl{\sigma_N}}(w,\beta).\label{E:Conditioned Szego Lift}
\end{equation}
We therefore start with the following
\begin{claim}
Fix $\xi \in \Sigma.$ In heisenberg coordinates centered at $\xi$ on $X,$ we have
\begin{equation}
\widehat{\sigma}_N(z, \alpha)\otimes \widehat{\cl{\sigma_N}}(w,\beta)= E_N(z,\alpha, w,\beta)^{-1}.\label{E:Sigma Lift v0}
\end{equation}
\end{claim}
\begin{proof}
Define the frame
\[e_{\xi}:=e^{\frac{N}{2}\lr{\gamma_N+i\twiddle{\gamma_N}}}\cdot \sigma_N\]
for $L^N$ near $\xi.$ We have, in the sense of Definition \ref{D:Preferred Frame}, that $e_{\xi}$ is a preffered frame for $L^N$ at $\xi.$ Therefore, in Heisenberg coordinates centered at $\xi$ on $X,$ we have
\[\widehat{\sigma_N}(z,\alpha)= \norm{e_{\xi}(z)}_he^{-\frac{N}{2}\lr{\gamma_N+i\twiddle{\gamma_N}}}\cdot e^{iN\alpha}.\]
Using that $\norm{e_{\xi}(z)}_h=e^{-\frac{N}{2}\lr{\gamma_N(z)-\phi_N(z)}},$ we conclude
\begin{equation}
  \label{E:Sigma Lift}
  \widehat{\sigma_N}(z,\alpha)=e^{-\frac{N}{2}\lr{\phi_N(z)+i\twiddle{\gamma_N}(z)}+iN\alpha}.
\end{equation}
Applying this formula to the lifts of $\sigma_N(z)$ and $\cl{\sigma_N}(w)$ to $X$ and taking their tensor product completes the argument.
\end{proof}
To verify (\ref{E:Conditioned Szego Lift}) it therefore remains to prove that 
\begin{equation}
\widehat{\Pi}_N^{\xi}(z,\alpha,w,\beta)= \widehat{\Pi}_N(z,\alpha, w,\beta)-\frac{\widehat{\Pi}_N(z,\alpha,\xi,0)\cl{\widehat{\Pi}_N(\xi,0,w,\beta)}}{\abs{\widehat{\Pi}_N(\xi,\alpha,\xi,\beta)}}.\label{E:Conditioned Lift II}
\end{equation}
This is precisely equation $(27)$ from \cite{Conditional}. For the reader's conveince, we reproduce the proof. To do this, we introduce, as in the proofs of Lemma 4 \S 8 of \cite{ZC} and Proposition 3.9 of \cite{Conditional}, the ``coherent state'' at $\xi.$ To do this,  
\[\Psi_N^{\xi}(z):=\frac{\Pi_{\sigma_N}(z,\xi)}{\Pi_{\sigma_N}(\xi,\xi)^{1/2}}\cdot \sigma_N(z) \in H_{hol}^0(\Sigma, L^N).\]
We recall from \S\ref{S:Def SZ} that $\Pi_{\sigma_N}$ is the \textit{unconditional} $N^{th}$ Bergman kernel relative to $\sigma_N,$ which we wrote as $\Pi_{\sigma_N}(z,w)=\sum_{j=1}^{d_n} f_j(z)\cl{f_j(w)}.$ Hence, 
\[\Psi_N^{\xi}(z)=\frac{1}{\lr{\sum_{j=1}^{d_N-1}\abs{f_j(\xi)}^2}^{1/2}}\sum_{j=1}^{d_N}\cl{f_j(\xi)} f_j(z)\cdot \sigma_N(z).\]
Using the weighted $L^2$ inner product (\ref{E:Inner Product}), we see that for every $s\in H_{hol}^0(\Sigma, L^N)=H_N$ satisfying $s(\xi)=0$
\[\inprod{s}{\Psi_N^{\xi}}=0\]
and that $\norm{\Psi_N^{\xi}}=1.$ Therefore, $\Psi_N^{\xi}$ spans the orthocomplement in $H_N$ to $H_N^{\xi}$ and 
\[\Pi_N^{\xi}(z,w)=\Pi_N(z,w)-\Psi_N^{\xi}(z)\otimes \cl{\Psi_N^{\xi}(w)}.\]
Lifting this equation to $X,$ (\ref{E:Conditioned Lift II}) reduces to showing that 
\[\widehat{\Psi}_N^{\xi}(z,\alpha)\otimes \widehat{\cl{\Psi_N ^{\xi}}}(w,\beta)=\frac{\widehat{\Pi}_N(z,\alpha,\xi,0)\widehat{\Pi}_N(\xi, 0, w,\beta)}{\widehat{\Pi}_N(\xi,\alpha,\xi,\beta)}.\]
To verify this equality, we note that, by formula (\ref{E:Sigma Lift}) for the lift of $\sigma_N,$ 
\[\widehat{\Psi}_N^{\xi}(z,\alpha)=\frac{1}{\sqrt{\sum \abs{f_j(\xi)}^2}}\sum \cl{f_j}(\xi)f_j(z) \cdot e^{-\frac{N}{2}\lr{\phi(z) + i \twiddle{\gamma_N}(z)}+iN\alpha}.\]
Hence,  
\[\widehat{\Psi}_N^{\xi}(z,\alpha)\otimes \widehat{\cl{\Psi}_N^{\xi}}(w,\beta)=\frac{\widehat{\Pi}_N(z,\alpha, \xi,0)E_N(z,\alpha,\xi,0)\widehat{\Pi}_N(\xi,0, w,\beta)E_N(\xi,0,w,\beta)}{\widehat{\Pi}_N(\xi,\alpha,\xi,\beta)E_N(\xi,\alpha, \xi,\beta)}E_N(z,\alpha, w,\beta)^{-1}.\]
Observing that 
\[\frac{E_N(z,\alpha,\xi,0)E_N(\xi,0,w,\beta)}{E_N(\xi,\alpha,\xi,\beta)}=E_N(z,\alpha, w,\beta)\]
completes the proof. 
\end{proof}
\end{proof}

\bibliographystyle{plain}
\bibliography{Bibliography_2}

\end{document}